\documentclass[12pt,a4paper]{amsart}

\usepackage{latexsym}
\usepackage{amssymb}
\newtheorem{prop}{Proposition}[section]

\newtheorem{thm}[prop]{Theorem}
\newtheorem{lemma}[prop]{Lemma}
\newtheorem{dfn}[prop]{Definition}
\newtheorem{rem}[prop]{Remark}
\newcommand{\sw}[1]{{}_{(#1)}}
\newcommand{\opm}{\cdot_\text{op}}
\newcommand{\op}{\text{op}}
\newcommand{\id}{\text{id}}

\begin{document}

\title{Semi--entwining structures and their applications}
\author{Florin F. Nichita}
\author{Deepak Parashar}
\author{Bartosz Zieli\'nski}
\address{Institute of Mathematics "Simion Stoilow" of the Rumanian Academy, P.O. Box 1-764, RO-014700 Bucharest, Rumania}
\email{Florin.Nichita@imar.ro}
\address{Cambridge Cancer Trials Centre,
Department of Oncology, University of Cambridge,
Addenbrookes Hospital (Box 279),
Hills Road,
Cambridge CB2 0QQ,
U.K.}
\address{MRC Biostatistics Unit Hub in Trials Methodology Research,
University Forvie Site,
Robinson Way,
Cambridge CB2 0SR,
U.K.}
\email{dp409@cam.ac.uk}
\address{Department of Theoretical Physics and Informatics, University of \L{}\'od\'z, Pomorska 149/153 90-236 \L{}\'od\'z, Poland}
\email{bzielinski@uni.lodz.pl}

\begin{abstract}
Semi--entwining structures are proposed as concepts simpler than
entwining structures, yet they are shown to have interesting applications in
constructing intertwining operators and 
braided algebras, lifting functors, finding
solutions for Yang-Baxter systems, etc.
While  for entwining structures one
can associate corings, for semi--entwining structures one can associate
comodule algebra structures
where the algebra involved is a bialgebra satisfying certain properties.

\end{abstract}
\footnote{ISRN Algebra, vol. 2013, Article ID 817919. doi:10.1155/2013/817919}

\maketitle
\section{Introduction and preliminaries}
Quantum groups appeared as symmetries of integrable systems in 
quantum and statistical mechanics in the works of Drinfeld and Jimbo. They
led to intensive studies of Hopf algebras from a purely algebraic 
point of view and to the development of more general 
categories of Hopf-type modules (see \cite{mp} for a recent review). These serve as 
representations of Hopf algebras and related structures, 
such as those described by the solutions to the Yang-Baxter equations.

Entwining structures were introduced in \cite{brzmaj} as generalized
symmetries of non-commutative principal bundles,
and provide a unifying framework for various Hopf--type modules.
They are related to the so called {\it mixed distributive laws} introduced in \cite{beck}.

The
Yang--Baxter systems emerged as spectral--parameter
independent generalization of the quantum Yang--Baxter equation related to
non--ultra-local integrable systems \cite{hlav, hs}.
Interesting links between the
entwining structures and Yang--Baxter systems
have been established in \cite{fbrz} and \cite{barbu}.
Both topics have been a focus of recent research (see,
for example, \cite{np1, np2, calin, wisb, caen, Shauen}).

\bigskip

In this paper,
we propose the concepts of {\em semi--entwining} structures and 
{\em cosemi--entwining} structures within a generic framework incorporating results of other authors alongside ours.
The semi--entwining structures are some kind of
entwining structures between an algebra and a module which obey only one half of their axioms,
while cosemi--entwining structures  are kind of entwining structures between a coalgebra and a module obeying the other half of their axioms.
The main motivations for this terminology are: the new constructions
which require only the axioms selected by us
(constructions of
intertwining operators and Yang-Baxter systems of type II, or liftings of
functors), 
our new examples of semi-entwining structures,
simplification of the work with certain structures (Tambara bialgebras,
lifting of functors, 
braided algebras and
Yang-Baxter systems of type I),
the connections of the category of semi--entwining structures with 
other categories, and so forth. Let us observe that while  for entwining structures one
can associate corings, for semi--entwining structures one can associate
comodule algebra structures
provided the algebra involved is a bialgebra with certain properties (see 
Theorem
\ref{teo}). 

\bigskip   

The current paper is organised as follows. Section 2 contains the newly
introduced terminology with examples, new results and comments. Section 3 is about
some of the applications of these concepts, namely, new
constructions of intertwining operators and  braided algebras, lifting
functors, and the
presentations of Tambara bialgebras
and
of (new families of) Yang--Baxter systems (of type I and II).

The main results of our paper are Theorems 3.1, 3.4, 3.6, 3.22, and 3.23.
Theorems 3.11 and 3.13 are mentioned in the context of stating some of our results.
Theorem 3.16 is used to prove Theorem 3.22, while
Theorem 3.19 is related to Theorem 3.20.

\bigskip

Unless otherwise stated, we work over a commutative ring $R$.
Unadorned tensor products mean tensor products over
$R$. 

For any $R$-module $V$, $T(V)$ denotes tensor algebra of $V$.
In Subsection 3.5, we work over a field $\mathbb{K}$.
For $ V $ an $R$-module, we denote by
$ \ I : V \rightarrow V $
the identity map.
For any $R$-modules $ V $ and $ W $ we denote by
$ \  \tau = \tau_{V, W}: V \otimes W \rightarrow W \otimes V \  $ the
twist map,
defined by $ \tau_{V, W}(v \otimes w) = w \otimes v $.
Let $ \  \phi: V \otimes V \rightarrow V \otimes V  $
be an $R$-linear map.
We use the following notations:
${\phi_{12}}= \phi \otimes I , \  {\phi_{23}}= I \otimes \phi , \
{\phi_{13}}=(I \otimes \tau_{V, V})(\phi \otimes I)(I\otimes \tau_{V, V})$.

\begin{dfn}
An invertible  $ R$-linear map  $ \phi : V \otimes V \rightarrow
V \otimes V $
is called a Yang--Baxter
operator if it satisfies
\begin{equation}  \label{ybeq}
\phi_{12}  \circ  \phi_{23}  \circ  \phi_{12} = \phi_{23}  \circ  \phi_{12}  \circ  \phi_{23}
\end{equation}
\end{dfn}

\begin{rem}
Equation (\ref{ybeq}) is usually called the braid equation. It is a
well--known fact that the operator $\phi$ satisfies (\ref{ybeq}) if and only if
$\phi\circ \tau_{V, V}  $ satisfies
  the quantum Yang--Baxter equation
(if and only if
$ \tau_{V, V} \circ \phi $ satisfies
the quantum Yang--Baxter equation):
\begin{equation}   \label{ybeq2}
\phi_{12}  \circ  \phi_{13}  \circ  \phi_{23} = \phi_{23}  \circ  \phi_{13}  \circ  \phi_{12}
\end{equation}
\end{rem}

\bigskip

\section{Semi--entwining structures and related structures}

\bigskip

\begin{dfn}
{\bf [Semi--entwining structures]}\label{semientdef}
Let $A$ be an $R$-algebra, and let $B$ be an $R$-module, then the $R$-linear map
$\psi:B\otimes A\rightarrow A\otimes B$ is called a (right) semi--entwining map if it satisfies the following
conditions for all $a,a'\in A$, $b\in B$
(where we use a Sweedler-like summation notation
$\psi(b\otimes a)=a_\alpha\otimes b^\alpha$):
\begin{gather}
\psi(b\otimes 1_A)=1_A\otimes b\nonumber\\
\psi(b\otimes aa')=a_\alpha {a'}_\beta\otimes b^{\alpha\beta}\label{semi}
\end{gather}
If $B$ is also an $R$-algebra, and a semi--entwining map satisfies additionally
$$
\psi(1_B\otimes a)=a\otimes 1_B, \ \ \ 
\psi(bb'\otimes a)=a_{\alpha\beta}\otimes b^\beta {b'}^\alpha\label{ent},
 \ \ \forall a\in A, \  \forall b, b'\in B
$$
then the semi--entwining map is called an algebra factorization (in the sense of \cite{brz}).

If $B$ is a coalgebra and satisfies
\begin{gather*}
a_\alpha\varepsilon(b^\alpha)=a\varepsilon(b),\\
a_\alpha\otimes b^\alpha\sw{1}\otimes b^\alpha\sw{2}
=a_{\alpha\beta}\otimes b\sw{1}{}^\beta\otimes b\sw{2}{}^\alpha,
\ \  \ \forall a\in A, \  \forall b\in B,
\end{gather*}
then $\psi$ is called a (left-left) entwining map \cite{brzmaj}.
\end{dfn}

\begin{rem} Let $q\in R$. The following are examples of semi--entwining structures. Note that they do not have natural algebra factorization structures
in general.
\begin{enumerate}
\item Let $A$ be an $R$-algebra, then the $R$-linear map
$\gamma_q :A\otimes A\rightarrow A\otimes A$,
$\gamma_q (b \otimes a) = 1 \otimes ba + q ba \otimes 1 - q b \otimes a $
is a semi--entwining map.

Notice that $\gamma_q $ is a Yang-Baxter operator (according to \cite{florin}).

\item Let $A$ be an $R$-algebra, then the $R$-linear map
$\eta_q :A\otimes A\rightarrow A\otimes A$,
$\eta_q (b \otimes a) =  q (ba-ab) \otimes 1 + a \otimes b $
is a semi--entwining map.

Notice that $\eta_q $ is a Yang-Baxter operator related to
Lie algebras (see, e.g.,
\cite{n}).

\item Let $A$ be an $R$-algebra and let $M$ be a right $A$-module.
Then the $R$-linear map
$\phi :M\otimes A\rightarrow A\otimes M$,
$\phi (m \otimes a) =  1 \otimes ma $
is a semi--entwining map.
\end{enumerate}
\end{rem}

The proof of the next lemma is direct; the second statement is 
a well-known result.

\begin{lemma}\label{algfact}
If $\psi:B\otimes A\rightarrow A\otimes B$ is a semi--entwining map,
then

(i) $ A \otimes B $ becomes a right $ A $-module with the operation
$ (a\otimes b) * a'= aa'{}_\alpha\otimes b^\alpha$;

(ii) moreover, if $B$ is an algebra, we can define
a bilinear operation
\begin{gather*}
\cdot:(A\otimes B)\otimes (A\otimes B)\rightarrow (A\otimes B),\\
(a\otimes b)\otimes (a'\otimes b')\mapsto aa'{}_\alpha\otimes b^\alpha b',
\end{gather*}
and $\cdot$ is an associative and unital multiplication on $A\otimes B$ if and only if
 $\psi$ is an algebra factorization.
\end{lemma}

\begin{rem}
Some authors call the above map $\psi$ a twisting map; see, for example,
\cite{ppv}, where a unifying framework for various twisted algebras is provided.
\end{rem}

\begin{rem}
Suppose that $A$ is a right $H$-comodule algebra (where $H$ is a bialgebra)
and $B$ is a right $H$-module. Then
\begin{equation}\label{entfor}
\psi_H:B\otimes A\rightarrow A\otimes B,\ b\otimes a\mapsto a\sw{0}\otimes ba\sw{1}
\end{equation}
is a semi--entwining map. Moreover, if $B$ is an $H$-module
algebra, then $\psi_H$ thus defined is an algebra factorization.
Finally, if $B$ is an $H$-module coalgebra, then $\psi$ is an entwining map,
and
 $(A,H,B)$ is called a Doi-Koppinen structure (see \cite{Shauen}).
\end{rem}

\begin{rem}\label{eq}
Let $A$ be an $R$-algebra.
We define the category of semi--entwining structures over $A$, whose objects are
triples $(B, \  A,\  \phi) $, and morphisms $ f: (B, \  A,\  \phi)
\rightarrow (B', \  A,\  \phi')$ are $R$-linear maps
$ f: B \rightarrow B'$ satisfying the relation
$ (I_A \otimes f) \circ \phi = \phi' \circ (f \otimes I_A)$.
Then, there exist the following functors.

\begin{enumerate}
\item F : Mod $A \ \rightarrow$ Semi--Entwining Str $A$

 $M \mapsto (M, \  A,\  \phi) $,
where $ \phi :M\otimes A\rightarrow A\otimes M, \ \
\phi (m \otimes a) =  1 \otimes ma $;

\item G :  Semi--Entwining Str $A$ $\ \rightarrow$ Mod $A$

$ (B, \ A, \ \psi) \mapsto  A \otimes B $,
where
$ A \otimes B $ is a right $ A $-module with the operation
$ (a\otimes b) * a'= aa'{}_\alpha\otimes b^\alpha$.

\end{enumerate}

These two functors do not form an equivalence of categories in general,
because $ F \circ G \simeq A \otimes - $ and  $ G \circ F \simeq A \otimes -$.

\end{rem}

\begin{thm}\label{teo}

If $\psi:B\otimes A\rightarrow A\otimes B$ is a semi--entwining map
and
$A$ is bialgebra, then:

\begin{enumerate}

\item $B$ is an $A$-bimodule with the following actions:

$a \circ b = \epsilon(a) b, \ \ \
b * a = \epsilon(a_{\alpha}) b^{\alpha}, \ \  
\forall a \in A, \ \forall b \in B.$

\item $B \oplus A$ is an algebra with 
the unit $(0,\ 1)$ and
the product 

$(b, \ a) \ (b', \ a') \ = \ ( b*a'+ a \circ b',\  aa')$,

and a right $A$-comodule with the coaction
$ \ \  b \oplus a \mapsto b \otimes 1 + (\sum a_1 \otimes a_2)$.

\item If $A$ has a bilateral integral (i.e.,$ \ \  ax=xa=\epsilon(a)x \ \ 
\forall a \in A$)
which is a group-like element (i.e.,$ \  \Delta(x)= x \otimes x, \ \
\epsilon(x)=1 $),
then  $B \oplus A$ is an $A$-comodule algebra with the coaction

$ b \oplus a \mapsto b \otimes x + (\sum a_1 \otimes a_2)$.

\end{enumerate}
\end{thm}

{\bf Proof.} 
(1) follows from the linearity of $ \epsilon $ and $ \psi $.

(2) follows from the previous statement and from direct computations as follows:

$      b \oplus a \mapsto b \otimes 1 + (\sum a_1 \otimes a_2) \ $
 maps to either
$      b \otimes (1 \otimes 1) + (\sum a_1 \otimes (a_{21} \otimes a_{22}))$
                  (if we apply the comultiplication of the algebra), 
or to

$(b \otimes 1) \otimes 1 + (\sum 0 \otimes 0) \otimes 1 + 0 \otimes 1 +
(\sum (a_{11} \otimes a_{12}) \otimes a_2) $
(if we apply the coaction).

We observe that the two outputs are equal.

(3) is a generalisation of (2), and is
left to the reader.\qed

\bigskip

Similarly we have the dual notion as follows.

\begin{dfn}{\bf [Cosemi--entwining structures]}
Let $C$ be an $R$-coalgebra, and $D$ an $R$-module.
A $R$-linear map $\psi:D\otimes C\rightarrow C\otimes D$ is called a cosemi--entwining map
if it satisfies the following
conditions for all $c,c'\in C$, $d\in D$
(where we use a Sweedler-like summation notation
$\psi(d\otimes c)=c^\alpha\otimes d_\alpha$):
\begin{gather*}
\varepsilon(c^\alpha)d_\alpha=\varepsilon(c)d,\\
c^\alpha\sw{1}\otimes c^\alpha\sw{2}\otimes d_\alpha=
c\sw{1}{}^\alpha\otimes c\sw{2}{}^\beta\otimes d_{\alpha\beta}.
\end{gather*}
If $D$ is also a coalgebra and $\psi$ satisfies additionally
\begin{gather*}
c^\alpha\varepsilon(d_\alpha)=c\varepsilon(d),\\
c^\alpha\otimes d_\alpha\sw{1}\otimes d_\alpha\sw{2}
=c^{\alpha\beta}\otimes d\sw{1}{}_\beta\otimes d\sw{2}{}_\alpha \ \ \
\forall \  d \in D, \  \forall \ c\in C,
\end{gather*}
then $\psi$ is called a coalgebra factorization.

If, on the other hand, $D$ is an algebra, and $\psi$ satisfies additionally
$$
\psi(1_D\otimes c)=c\otimes 1_D, \ \ \
c^\alpha \otimes (dd')_\alpha=c^{\alpha\beta}\otimes d_\beta d'{}_\alpha, \ \ \
\forall \  d,d'\in D, \  \forall \ c\in C, $$
 then $\psi$ is called a (right-right)  entwining map.
\end{dfn}

The next result is dual to Lemma \ref{algfact}.

\begin{lemma}\label{coalgfact}
Suppose that $\psi:D\otimes C\rightarrow C\otimes D$  is a cosemi--entwining map and
$D$ is a coalgebra. Define a map
\begin{gather*}
\Delta_{D\otimes C}:(D\otimes C)\rightarrow (D\otimes C)\otimes (D\otimes C),\\
d\otimes c\mapsto (d\sw{1}\otimes c\sw{1}{}^\alpha)\otimes (d\sw{2}{}_\alpha\otimes c\sw{2}).
\end{gather*}
Then $\Delta$ makes $D\otimes C$ a coalgebra if and only if $\psi$ is a coalgebra
factorization.
\end{lemma}
\begin{proof}
For $D\otimes C$ to be a coalgebra it must satisfy the counit property, i.e., 
$(\varepsilon_{D\otimes C}\otimes\id)\circ\Delta_{D\otimes C}=(\id\otimes\varepsilon_{D\otimes C})\circ\Delta_{D\otimes C}=\id$
and the  coassociativity property. To check a counit property note that for all $d\in D$ and $c\in C$:
$$
(\id\otimes\varepsilon_{D\otimes C})\circ\Delta_{D\otimes C}(d\otimes c)=d\sw{1}\otimes c^\alpha\varepsilon(d\sw{2}{}_\alpha).
$$
Now, if $d\otimes c=d\sw{1}\otimes c^\alpha\varepsilon(d\sw{2}{}_\alpha)$, then applying $\varepsilon\otimes\id$ to both
sides of this equation yields $c\varepsilon(d)=c^\alpha\varepsilon(d_\alpha)$. Similarly,
we prove the other half of the counit property. Conversely, $c\varepsilon(d)=c^\alpha\varepsilon(d_\alpha)$
implies the counit property.

Using the fact that $\psi$ is a cosemi-entwining map it is easy to prove that the coassociativity
implies that for all $c\in C$ and $d\in D$
$$
c\sw{1}{}^\alpha\otimes d_\alpha\sw{1}\otimes c\sw{2}{}^\beta\otimes d_\alpha\sw{2}{}_\beta
=c\sw{1}{}^{\alpha\gamma}\otimes d\sw{1}{}_\gamma\otimes c\sw{2}{}^\beta\otimes d\sw{2}{}_{\alpha\beta}
$$
Applying $\varepsilon$ to the third leg and using the fact that $\psi$ is a cosemi-entwining map
yields
$$
c^\alpha\otimes d_\alpha\sw{1}\otimes  d_\alpha\sw{2}
=c^{\alpha\gamma}\otimes d\sw{1}{}_\gamma\otimes d\sw{2}{}_{\alpha}
$$
We leave the rest of the proof to the reader.
\end{proof}

\begin{rem}
Suppose that $C$ is a right $H$-comodule coalgebra (where $H$ is a bialgebra)
and $D$ is a right $H$-module. Then
\begin{equation}\label{altdoi}
\psi :D\otimes C\rightarrow C\otimes D,\ d\otimes c\mapsto c\sw{0}\otimes dc\sw{1}
\end{equation}
is a cosemi--entwining map. Furthermore, if $D$ is an $H$-module coalgebra, then $\psi$ is a coalgebra factorization.
Otherwise, if $D$ is an $H$-module algebra, then $\psi$ is a left-left entwining map.
Moreover, in this last case, $(C,H,D)$ is called an alternative  Doi-Koppinen structure.
\end{rem}

Let $X$, $Y$ be any $R$-modules. Any $x^*\in X^*$ can be viewed as the map
\begin{equation*}
x^*:X\otimes Y\rightarrow Y,\ x\otimes y\mapsto x^*(x)y.
\end{equation*}
Also any tensor $\sum_i x^*_i\otimes y_i\in X^*\otimes Y$ can be considered as a map
$X\ni x\mapsto\sum_ix^*_i(x)y_i\in Y$.
Finally, if $X$ is finitely generated and projective,
then $\text{Hom}_R(X,Y)\sim X^*\otimes Y$.
For any $y\in Y$, an $R$-module map $\Psi:Y\otimes X\rightarrow X\otimes Y$ defines
a map
\begin{equation*}
\Psi_y=\Psi(y\otimes\cdot):X\rightarrow X\otimes Y.
\end{equation*}
We define a dual of $\Psi^{*_X}:Y\otimes X^*\rightarrow X^*\otimes Y$ with respect to the
$X$-part as $\Psi^{*_X}(y\otimes x^*)=\Psi^*_y(x^*)$, where
$\Psi_y^*:X^*\rightarrow X^*\otimes Y$ is defined by
\begin{equation*}
x^*(\Psi_y(x))=\psi_y^*(x^*)(x),\ \text{for all }x\in X,x^*\in X^*, y\in Y.
\end{equation*}
Similarly, one defines a dual $\Psi^{*_Y}:Y^*\otimes X\rightarrow X\otimes Y^*$ of $\Psi$
with respect to the $Y$-part.

\bigskip

The next lemma is a standard result. 

\begin{lemma}
Suppose that $C$ is a finitely generated projective $R$-coalgebra, and
$(c_i\in C,c^*_i)$ is a dual basis.
Let $\psi:D\otimes C\rightarrow C\otimes D$ be a cosemi--entwining map.
Then $\psi^{*_C}:D\otimes C^*\rightarrow C^* \otimes D$ is a semi--entwining map for the convolution algebra
$C^*$.

Explicitly
\begin{equation*}
\psi^{*_C}(d\otimes c^*)=\sum_i c^*_i\otimes c^*(c_i^\alpha)d_\alpha.
\end{equation*}
\end{lemma}

\begin{dfn} {\bf [Semi--entwined modules and comodules]} Let $A$ be an algebra, and let $V$ be a vector space. Suppose that
$\psi:V\otimes A\rightarrow A\otimes V$ is a semi--entwining map and $M$ a right $A$-module.
\begin{enumerate}
\item Let $\triangleleft:M\otimes V\rightarrow M$
be a right measuring, such that, for all $m\in M$, $a\in A$, $v\in V$,
\begin{equation*}
ma_\alpha\triangleleft v^\alpha=(m\triangleleft v)a.
\end{equation*}
Then $M$ is called a $(A,V,\psi)$-semi--entwined module.
\item Let $\rho:M\rightarrow M\otimes V$, $m\rightarrow m\sw{0}\otimes m\sw{1}$ be a right comeasuring, such that, for all
$m\in M$, $a\in A$,
\begin{equation*}
\rho(ma)=m\sw{0}\psi(m\sw{1}\otimes a).
\end{equation*}
Then $M$ is called a $(A,V,\psi)$-semi--entwined comodule.
\end{enumerate}
\end{dfn}

\begin{rem} The following are examples of
semi--entwining modules related to Remark 2.2:
\begin{enumerate}
\item let $A$ be an $R$-algebra, let $M$ be a right $A$ module, $V$ = $A$,
$\psi = \gamma_q$, and the right measuring the regular action of $A$
on $M$;
\item let $A$ be an $R$-algebra, let $M$ be a right $A$ module, $V$ = $A$,
$\psi = \eta_1$, and the right measuring the regular action of $A$
on $M$.
\end{enumerate}
\end{rem}

\begin{rem} The following are examples of
semi--entwining comodules related to Remark 2.2:
\begin{enumerate}
\item let $A$ be an $R$-algebra, let $M$ be a right $A$ module, $V$ = $A$,
$\psi = \gamma_1$, and the right comeasuring $ \rho(m) = m \otimes 1 $;
\item let $A$ be an $R$-algebra, let $M$ be a right $A$ module, $V$ = $A$,
$\psi = \eta_q$,
and the right comeasuring $ \rho(m) = m \otimes 1 $.
\end{enumerate}
\end{rem}

\bigskip

\newcommand{\swb}[1]{{}_{\overline{#1}}}

\begin{dfn}{\bf [Cosemi--entwined modules and comodules]} Let $C$ be a coalgebra, and $V$ a vector space. Suppose that
$\psi:V\otimes C\rightarrow C\otimes V$ is a cosemi--entwining map and $M$ a left $C$-comodule, with a coaction ${}^C\rho:M\rightarrow C\otimes M$,
$m\mapsto m\sw{-1}\otimes m\sw{0}$.
\begin{enumerate}
\item Let $\triangleright:V\otimes M\rightarrow M$
be a left measuring, such that, for all $m\in M$, $v\in V$,
\begin{equation*}
{}^C\rho(v\triangleright m)=m\sw{-1}{}_\alpha\otimes v^\alpha\triangleright m\sw{0}.
\end{equation*}
Then $M$ is called a $(C,V,\psi)$-cosemi--entwined module.
\item Let ${}^V\rho:M\rightarrow V\otimes M$, $m\mapsto m\swb{-1}\otimes m\swb{0}$ be a left comeasuring, such that for all $m\in M$,
\begin{equation*}
(\text{id}_C\otimes {}^V\rho)\circ{}^C\rho(m)=m\swb{0}\sw{-1}{}_\alpha\otimes m\swb{-1}{}^\alpha\otimes m\swb{0}\sw{0}.
\end{equation*}
Then $M$ is called a $(C,V,\psi)$-cosemi--entwined comodule.
\end{enumerate}
\end{dfn}

Note that if $V$ is a coalgebra and $\psi:V\otimes A\rightarrow A\otimes V$ is an entwining map, then a semi--entwined module $M$ is an entwined module.

The following result is standard, but we provide a partial proof for completeness.
\begin{lemma}
Suppose that $(A,B,\psi)$ is an algebra factorization, and $M$ is a $(A,B,\psi)$ semi--entwined module such that the $B$ measuring is an action.
Then $M$ is a right $A\otimes B$-module, with an algebra structure on
$A\otimes B$ as in Lemma~\ref{algfact}, and $A\otimes B$ action on $M$ given by $m(a\otimes b)=(ma)\triangleleft b$.
Conversely, any right $A\otimes B$ module is a semi--entwined $(A,B,\psi)$-module with $A$ and $B$ actions given
 by $ma=m(a\otimes 1_B)$ and $m\triangleleft b=m(1_A\otimes b)$, respectively.
\end{lemma}
\begin{proof}
It is enough to verify that the definition of $A\otimes B$ action agrees with the
algebra relations, i.e., that
$$
m((1\otimes b)(a\otimes 1))=(m(1\otimes b))(a\otimes b)
$$
Both sides of the above equation equal $ma_\alpha\triangleleft b^\alpha$ -- left one because
of algebra relations, and the right one because $M$ is a $(A,B,\psi)$ semi--entwined module.
We prove similarly the rest of the lemma. 
\end{proof}

\bigskip

\section{ Applications}

\bigskip

\subsection{Intertwining Operators}

\bigskip

We give a brief introduction to the intertwining operators below.

Let $A$ be an $R$-algebra.
Given two algebra representations, say
$ \rho: V \otimes A \rightarrow V$
and
$ \rho': V' \otimes A \rightarrow  V'$,
we define an { \bf intertwining operator} $ f:V \rightarrow V'$
to be a linear operator such that
$f \circ \rho = \rho' \circ (f \otimes I)$.

With this definition we can define the category of finite dimensional representations of $A$, in which the morphisms are intertwining operators
(see \cite{bl}).

\bigskip

The following theorem provides a connection between semi-entwining structures
and intertwining operators.

\bigskip

\begin{thm}\label{int}

Let $A$ be an $R$-algebra, let $B$ be an $R$-module, and let
$\psi:B\otimes A\rightarrow A\otimes B$ be a semi--entwining map.
Then, the following statements are true:

(i) $B\otimes A $ is a right $A$-module in a trivial way,
with the right action $ \rho: (B\otimes A)\otimes A \rightarrow (B\otimes A), \
\ \ (b\otimes a)\otimes a' \mapsto b \otimes aa'$.

(ii) $A\otimes B$ is a right $A$-module in the following way:
$ \rho': (A\otimes B)\otimes A \rightarrow (A\otimes B), \
\ \ (a\otimes b)\otimes a' \mapsto a a'_{\alpha} \otimes b^{\alpha}$.

(iii) With the above actions, $\psi:B\otimes A\rightarrow A\otimes B$
is an intertwining operator (i.e. $\psi$ satisfies the relation
$\psi \circ \rho = \rho' \circ (\psi \otimes I)$).

\end{thm}

{\bf Proof.} The proof of (i) is direct and (ii) follows from Lemma 2.3(i).
The relation $\psi \circ \rho = \rho' \circ (\psi \otimes I)$ is equivalent
to the second relation of (3).
\qed

\bigskip

\subsection{Braided Algebras}

Many algebras obtained by
quantization are commutative braided algebras, and
all super-commutative algebras are automatically
commutative braided algebras (see \cite{bb}).

\begin{dfn}
An algebra $ (A, M, u)$ for which there exists a Yang-Baxter operator

$ \psi : A \otimes A \rightarrow  A \otimes A $ such that
$ \psi( a \otimes 1)= 1 \otimes a $,  $ \ \ \psi( 1 \otimes b)= b \otimes 1 $,\\
$ \psi( a \otimes bc)= (M \otimes I) \circ ( I \otimes \psi) \circ
(\psi \otimes I)
( a \otimes b \otimes c) $  and

$ \psi( ab \otimes c)=
 (I \otimes M) \circ (\psi \otimes I) \circ (I \otimes \psi)
( a \otimes b \otimes c)$
$ \ \ \ \forall a,b,c \in A$

is called a braided algebra.

Moreover, if
$ M \circ \psi( a \otimes b)= M (a \otimes b) \ \  \forall a,b \in A$,
we call $ (A, M, u, \psi )$
a commutative
braided algebra or an r-commutative algebra (see \cite{b}).
\end{dfn}

\begin{dfn}
Given braided algebras $ (A, M, u, \psi )$ and  $ (B, M, u, \psi' )$,
we say that $ f : A \rightarrow B$ is a braided algebra morphism if it is
a morphism of algebras and $ ( f \otimes f) \circ \psi = \psi' \circ ( f \otimes f)$ (see \cite{b}).
\end{dfn}

\begin{thm}
(i) Any algebra $ (A, M, u)$ becomes a commutative
braided algebra
$ (A, M, u, \psi )$ with
$ \psi( a \otimes b)=\psi^A( a \otimes b)= 1 \otimes ab + ab \otimes 1 - a \otimes b$.

(ii) If $ (A, M, u, \psi^A )$ and  $ (B, M, u, \psi^B )$ are two braided
algebras as in (i), and $ f : A \rightarrow B$ is an algebra morphism, then it is also a
braided algebra morphism.

(iii) If $ \delta : A \rightarrow A$ is a derivation (i.e., $ \delta(ab) = \delta(a)b+
a \delta(b)  \  $ and $ \delta(1)=0$), then there exists a morphism of braided
algebras 

$ f: (A, M, u, \psi^A )  \rightarrow  (A \oplus A, m, \eta, \psi^{A \oplus A }),\ \  a \mapsto a \oplus \delta(a) $, 

where 
$ m( (a \oplus b) \otimes (a' \oplus b')) = (aa') \oplus (ab'+ba')$ and
$1_{A \oplus A} = 1_A \oplus 0_A$.

\end{thm}

{\bf Proof.}

(i) Notice that
$ \psi( a \otimes b)= 1 \otimes ab + ab \otimes 1 - a \otimes b$ is a self-inverse Yang-Baxter operator which was studied in \cite{nic, n}.

$ \psi( a \otimes 1)= 1 \otimes a $,  $ \ \ \psi( 1 \otimes b)= b \otimes 1 $
(directly)

$ \psi( a \otimes bc)= (M \otimes I) \circ ( I \otimes \psi) \circ
(\psi \otimes I)
( a \otimes b \otimes c) $  (from Remark 2.2 (i) with q=1)

$ \psi( ab \otimes c)= 1 \otimes abc + abc \otimes 1 - ab \otimes c =
 (I \otimes M) \circ (\psi \otimes I) \circ (I \otimes \psi)
( a \otimes b \otimes c)=  (I \otimes M) \circ (\psi \otimes I)(
a \otimes 1 \otimes bc + a \otimes bc \otimes 1 - a \otimes b \otimes c)=
(I \otimes M) ( 1 \otimes a \otimes bc + abc \otimes 1 \otimes 1 +
1 \otimes abc \otimes 1 -
 a \otimes bc \otimes 1 -
 1 \otimes ab \otimes c -  ab \otimes 1 \otimes c + a \otimes b \otimes c)=
1 \otimes abc + abc \otimes 1 +
1 \otimes abc \otimes 1 -
 a \otimes bc -
 1 \otimes abc -  ab \otimes c + a \otimes bc =
abc \otimes 1 -
 1 \otimes abc -  ab \otimes c$

$ M \circ \psi( a \otimes b)= 1 \otimes ab + ab \otimes 1 - a \otimes b =
ab = M (a \otimes b) $.

(ii) This follows from Proposition 3.1 of \cite{florin}. Also, refer to \cite{n}.

(iii) The proof is direct and is left to the reader.
\qed

\begin{rem} In
the above example $ \psi \circ \psi = I \otimes I $; so, the above algebra is
``strong''. All sorts of non-commutative analogs of manifolds are
commutative braided algebras:
quantum groups, non-commutative tori, quantum vector spaces,
the Weyl and Clifford algebras, certain universal enveloping algebras, super-manifolds, and so forth. It seems that the ones with direct relevance to quantum theory in 4 dimensions are ``strong,'' while the non-strong ones, like quantum groups, are primarily relevant to 2- and 3-dimensional physics (see \cite{bb}).

\end{rem}

\bigskip

\subsection{Liftings of Functors}

The semi-entwining structures can be understood as liftings of functors
from one category to another.
This goes back as far back as \cite{peter}. This situation is reviewed in
\cite{wisb}: the semi-entwining case is dealt with in general in item
3.3 (which is transfered from  \cite{peter}); how this general case is 
translated to our situation is clear from the discussion in item 5.8 of
\cite{wisb}. This is also presented in subsection 3.1 of \cite{wisbauer},
where the axioms of semi-entwining structures are given by formula (3.1). 

We give a general definition of liftings of functors. $F$ is a lifting of $G$ if the following diagram commutes:

\begin{picture}(100,100)(10,10)
\put(105,80){$ \mathfrak{C}  $ }
\put(105,10){$ \mathfrak{A} $ }
\put(111,74){\vector(0,-1){53}}
\put(245,80){$  \mathfrak{D} $ }
\put(245,10){$  \mathfrak{B} $ }
\put(250,74){\vector(0,-1){53}}
\put(128,86){\vector(1,0){104}}
\put(137,16){\vector(1,0){87}}
\put(172,92){$ F $}
\put(95,42){ $ U $ }
\put(248,42){ $ U' $ }
\put(174,22){$ G  $}
\end{picture}

where $U$ and $U'$ are forgetful functors.

We now present examples of liftings of functors related to
semi-entwining structures.

\begin{thm}\label{lift} Let $A$ be an $R$-algebra, and let $B$ be an $R$-module.
The functor $ \  - \otimes B$ can be
lifted from the category of $ R$-modules to
the category of right $A$-modules $\iff$ there exists a $R$-linear map
$\psi:B\otimes A\rightarrow A\otimes B$ which is a semi--entwining map.

\end{thm}

{\bf Proof.} Assume that there exists a semi--entwining
$\psi:B\otimes A\rightarrow A\otimes B$, then $-\otimes B$ lifts to a functor which associates to
a right $A$-module
$M$ the $A$-module $M\otimes B$ with  a right $A$ action given by
$$
(m\otimes b)a:=ma_\alpha\otimes b^\alpha.
$$
It remains to check that for any
 right $A$-module function $f:M\rightarrow M'$,
the map $f\otimes\id:M\otimes B\rightarrow M'\otimes B$ is a right $A$-module map:
$$
(f\otimes\id)(m\otimes b)a=
(f(m)\otimes b)a=f(m)a_\alpha\otimes b^\alpha=f(ma_\alpha)\otimes b^\alpha
=(f\otimes\id)((m\otimes b)a).
$$
On the other hand, suppose that $-\otimes B$ lifts to a functor in the category of right $A$-modules.
In particular, it follows that $A\otimes B$ is a right $A$-module. Define the linear map
$$
\Psi:B\otimes A\rightarrow A\otimes B,\quad b\otimes a\mapsto a_\alpha\otimes b^\alpha
$$
by the formula
$$
\Psi(b\otimes a):=(1\otimes b)a.
$$
We shall prove that this is a semi--entwining map. Indeed, by definition we have
$$\Psi(b\otimes 1)=1\otimes b.$$ Any element $a\in A$ defines  a right $A$-module map
$$f:A\rightarrow A,\quad a'\mapsto aa'.$$
It follows that for any $a'\in A$ we have from the $A$-linearity of $f\otimes\id$:
$$
(a\otimes b)a'=(f(1)\otimes b)a'=(f\otimes\id)((1\otimes b)a')=f(a'{}_\alpha)\otimes b^\alpha=aa'{}_\alpha\otimes
b^\alpha.
$$
Hence
$ \ \
(aa')_\alpha\otimes b^\alpha=(1\otimes b)(aa')
=(a_\alpha\otimes b^\alpha)a'=a_\alpha a'{}_\beta\otimes b^{\alpha\beta}. \ \
$
\qed

\begin{rem}
Let $A$ be an $R$-algebra and let $B$ be an $R$-module.
Using our terminology (given in Remark 2.6) and the results of \cite{wisbauer},
we conclude that the category of semi-entwining structures over $A$ is isomorphic to
the category of lifting of functors from the  
category of $ R$-modules to
the category of right $A$-modules.
\end{rem}

\bigskip

\begin{rem} We now give a more general
definition than that given in Remark \ref{eq}.

We define the category of semi--entwining structures, 
whose objects are
triples $(B, \  A,\  \phi) $, and morphisms are pairs 
$ (f,\ g): (B, \  A,\  \phi)
\rightarrow (B', \  A',\  \phi')$ 
where $ f: B \rightarrow B'$ is an
$R$-linear map, $ g: A \rightarrow A'$ is an algebra morphism,
and they satisfy
the relation
$ (g \otimes f) \circ \phi = \phi' \circ (f \otimes g)$.

In a dual manner, let us define
the category of cosemi--entwining structures,
whose objects are
triples $(D, \  C,\  \phi) $,
and morphisms are pairs 
$ (f,\ g): (D, \  C,\  \phi)
\rightarrow (D', \  C',\  \phi')$ 
where $ f: D \rightarrow D'$ is an
$R$-linear map, $ g: C \rightarrow C'$ is a coalgebra morphism,
and they satisfy
the relation
$ (g \otimes f) \circ \phi = \phi' \circ (f \otimes g)$.

The duality functor from the category of
coalgebras to the category of algebras
can be lifted 
to a functor from the category
of cosemi-entwining structures
to the 
category
of semi-entwining structures (by Lemma 2.11). 

This fact is described in the following diagram.

\begin{picture}(100,100)(10,10)
\put(10,80){$ \bf Cosemi-entw \ str  $ }
\put(70,10){$ \bf  \ k-coalg $ }
\put(111,74){\vector(0,-1){53}}
\put(240,80){$ \bf { Semi-entwining \ str } $ }
\put(230,10){$ \bf { \ k-alg} $ }
\put(250,74){\vector(0,-1){53}}
\put(128,86){\vector(1,0){104}}
\put(137,16){\vector(1,0){87}}
\put(172,92){$ { ( \ ) }^* $}
\put(95,42){ $ U $ }
\put(248,42){ $ U $ }
\put(174,22){$ { ( \ ) }^*   $}
\end{picture}

\end{rem}

\bigskip

\begin{rem} A braided coalgebra is a structure dual to 
Definition 3.2 (see, e.g. \cite{malte}). 
 
The duality between finite-dimensional 
algebras and finite-dimensional coalgebras can be lifted 
to a duality between the categories
of  finite-dimensional braided
algebras and  finite-dimensional braided coalgebras. 
This fact is described in the following diagram.

\begin{picture}(100,100)(10,10)
\put(30,80){$ \bf f.d. \ braided \  alg $ }
\put(70,10){$ \bf f.d. \ k-alg $ }
\put(111,74){\vector(0,-1){53}}
\put(240,80){$ \bf {f.d. \ braided \  coalg } $ }
\put(230,10){$ \bf {f.d. \ k-coalg} $ }
\put(250,74){\vector(0,-1){53}}
\put(128,86){\vector(1,0){104}}
\put(232,80){\vector(-1,0){104}}
\put(137,16){\vector(1,0){87}}
\put(223,10){\vector(-1,0){87}}
\put(172,92){$ { ( \ ) }^* $}
\put(172,70){$ { ( \ ) }^* $}
\put(95,42){ $ U $ }
\put(248,42){ $ U $ }
\put(174,22){$ { ( \ ) }^*   $}
\put(174,0){$ { ( \ ) }^*   $}
\end{picture}

\end{rem}

\bigskip

\subsection{Tambara bialgebras}

\begin{dfn}
{\bf [Tambara Bialgebra (\cite{Tambara})]}\label{tambdef}
Let $A$ be a finitely generated and projective $R$-algebra
(which implies that $A^*$ is a coalgebra), and let $a_i$,
$a^*_i$, $i=1,\ldots, N$ be a dual basis of $A$. Let $I\subset T(A^*\otimes A)$
be an ideal generated by elements
\begin{gather*}
a^*(1_A)-a^*\otimes 1_A,\\
a^*\otimes aa'-a^*\sw{1}\otimes a\otimes a^*\sw{2}\otimes a',
\end{gather*}
for all $a\in A$, $a^*\in A^*$.
Then $H(A)=T(A^*\otimes A)/I$ is called a Tambara bialgebra. Denoting by $[a^*\otimes a]$ the
class of $a\otimes a^*$ in $H(A)$,
the comultiplication $\Delta$ and counit $\varepsilon$ is given by
\begin{gather*}
\Delta([a^*\otimes a])=[\sum_i a^*\otimes a_i]\otimes [a^*_i\otimes a],\\
\varepsilon([a^*\otimes a])=a^*(a).
\end{gather*}
$A$ is a right $H(A)$-comodule algebra with coaction
\begin{gather*}
\varrho(a)=\sum_ia_i\otimes [a^*_i\otimes a].
\end{gather*}
\end{dfn}

\begin{thm}{\bf (\cite{Tambara})} \label{Tamb}
Suppose that $A$ is a finitely generated projective $R$-algebra and $B$ an $R$-module. Then
semi--entwining structures $\psi:B\otimes A\rightarrow A\otimes B$ are in one
to one correspondence with right $H(A)$-module structures on $B$. Similarly,
if $B$ is an algebra then algebra factorizations are in one to one correspondence
with right $H(A)$-module algebra structures on $B$.
Finally if $B$ is a coalgebra, then entwining structures $\psi:B\otimes A\rightarrow A\otimes B$ are in one to one correspondence with
right $H(A)$-module coalgebra structures on $B$.
 Explicitly, given
right $H(A)$-module structure on $B$, we define $\psi=\psi_{H(A)}$ (eq. \ref{entfor}).
Conversely, given a semi--entwining $\psi:B\otimes A\rightarrow A\otimes B$, we define
a right $H(A)$ module action on $B$ by
\begin{equation}
b[a^*\otimes a]=a^*(a_\alpha)b^\alpha.
\end{equation}
\end{thm}

\begin{rem} Let $q\in R$. The  examples of semi--entwining structures
presented in Remark 2.2 generate the following structures:

\begin{enumerate}

\item
 a right $H(A)$ module action on $A$ by

$b[a^*\otimes a]=a^*(1)ba + q a^*(ba) 1_A - q  a^*(b)a$;

\item
 a right $H(A)$ module action on $A$ by

$b[a^*\otimes a]= q a^*(ba-ab) 1_A + q  a^*(a)b$;

\item
 a right $H(A)$ module action on $M$, for any right A-module $M$, by

$m[a^*\otimes a]=  a^*(1_A)ma $.

\end{enumerate}
\end{rem}

Let $C$ be a finitely generated and projective $R$-coalgebra. Let $c_i$, $c_i^\ast$, $i=1,\ldots,N$ be a dual basis of $C$. Note
that $H(C^\ast)^{\text{cop}}=T(C^\ast\otimes C)/I'$ where $I'\subset T(C^\ast\otimes C)$ is an ideal generated by elements
\begin{gather*}
\varepsilon_C(c)-\varepsilon_C\otimes C, \\
c^\ast\ast d^\ast\otimes c-(c^\ast\otimes c\sw{1})\otimes (d^\ast\otimes c\sw{2}),
\end{gather*}
for all $c^\ast,d^\ast\in C^\ast$, $c\in C$, with explicit coaction and counit given by
\begin{gather*}
\Delta([c^*\otimes c])=[\sum_i c^*\otimes c_i]\otimes [c^*_i\otimes c],\\
\varepsilon([c^*\otimes c])=c^*(c).
\end{gather*}

\begin{thm}{\bf (\cite{Tambara})} \label{Tamb2}
Suppose that  $C$ is a finitely generated projective $R$-coalgebra, and $D$ an $R$-module. Then cosemi--entwining structures
$\psi:D\otimes C\rightarrow C\otimes D$ are in one-to-one correspondence with right $H(C^\ast)^{\text{cop}}$-module structures on $D$. Similarly
if $D$ is a coalgebra, then coalgebra factorizations are in one to one correspondence with $H(C^\ast)^{\text{cop}}$-module coalgebra structures on $D$.
Finally, if $D$ is an algebra, then (right-right) entwining structures $\psi:D\otimes C\rightarrow C\otimes D$ are in one to one
correspondence with right $H(C^\ast)^{\text{cop}}$-module algebra structures on $D$.
Explicitly, given right $H(C^\ast)^{\text{cop}}$-module structures on $D$, we define $\psi=\psi^{H(C^\ast)^{\text{cop}}}$ (eq. ~\ref{altdoi}). Conversely, given a
cosemi--entwining $\psi:D\otimes C\rightarrow C\otimes D$, we define a right $H(C^\ast)^{\text{cop}}$-module structures on $D$ by
\begin{equation}
d[c^*\otimes c]=c^*(d_\alpha)c^\alpha.
\end{equation}
\end{thm}

\bigskip

\bigskip

\subsection{Yang--Baxter Systems}

From now on we work over a field $\mathbb{K}$. It is convenient to introduce the {\em constant Yang--Baxter commutator} of the linear maps

$R:V\otimes V'\rightarrow V\otimes V',\ S:V\otimes V''\rightarrow
V\otimes V'',\
T:V'\otimes V''\rightarrow V'\otimes V''$ by
$$ [R,S,T]:= R_{12} S_{13} T_{23} - T_{23} S_{13} R_{12}. $$

In this notation, the quantum Yang--Baxter equation reads
$  [R,R,R]=0.$

\bigskip

\begin{dfn}{\bf [Yang--Baxter systems of type I]}
A system of linear maps of vector spaces
$
W:V\otimes V\rightarrow V\otimes V,\ Z:V'\otimes V'\rightarrow V'\otimes V',\
X:V\otimes V'\rightarrow V\otimes V' \
$
is called a WXZ system (or a Yang--Baxter system of type I) if
\begin{subequations}
\label{WXZ}
\begin{gather}
[W,W,W]=0,\ [W,X,X]=0,\label{WXZa}\\
[Z,Z,Z]=0,\ [X,X,Z]=0.\label{WXZb}
\end{gather}
\end{subequations}
A system of linear maps $W,X$ satisfying equations (\ref{WXZa})
is called a semi
Yang--Baxter system. One can associate a  WXZ system to a semi
Yang--Baxter system by setting $ Z = I \otimes I$.
\end{dfn}

\begin{rem}\label{remyb1}
From a Yang--Baxter system of type I, one can construct
a Yang--Baxter operator on $(V \oplus V)\otimes (V\oplus V)$, provided that the map $ X $ is invertible
(see \cite{fbrz}).
\end{rem}

\bigskip

Let $A$ be an algebra, and the map
\begin{equation}\label{FlR}
W=R^A_{r,s}:A\otimes A\rightarrow A\otimes A,\
a\otimes b\mapsto sba\otimes 1+ r1\otimes ba -sb\otimes a,
\end{equation}
for some arbitrary $s,r\in\mathbb{K}$ (see \cite{florin}).
Then, $[W,W,W]=0$.

The following is an enhanced version of Theorem 2.3 of \cite{fbrz}.

\begin{thm}(see \cite{fbrz})\label{FlorinBrze}
Let $A$ be an algebra, let $B$ be a vector space, and
$p,q,s,r\in\mathbb{K}$.

Let $W=R^A_{r,s}$, and let $X:A\otimes B\rightarrow A\otimes B$ be a linear map
such that $X(1_A\otimes b)=1_A\otimes b$, for all $b\in B$.

i) Then
$W,X$ is a semi Yang--Baxter system if and only if $\psi=X\circ \tau_{B,A}$ is a
semi--entwining map.

ii) Similarly, if $B$ is an algebra,
$Z=R^B_{p,q}$, and $X(a\otimes 1_B)=a\otimes 1_B$ for all $a\in A,$
then $W,X,Z$ is a Yang--Baxter system of type I if and only if $\psi$ is an algebra
factorization.
\end{thm}

\bigskip

\begin{dfn}{\bf [Yang--Baxter systems of type II]}
A system of linear maps of vector spaces
$ \ \mathbb{A}, \ \mathbb{B}, \ \mathbb{C}, \ \mathbb{D} : V \otimes V
\rightarrow V \otimes V$
is called  a Yang--Baxter system of type II if
$$ [\mathbb{A},\mathbb{A},\mathbb{A}]=0,
\ \ \ \ \ \ \ [\mathbb{D},\mathbb{D},\mathbb{D}]=0,$$
$$ [\mathbb{A},\mathbb{C},\mathbb{C}]=0, \ \ \ \ \ \ \
[\mathbb{D},\mathbb{B},\mathbb{B}]=0,$$
$$ [\mathbb{A},\mathbb{B^+},\mathbb{B^+}]=0, \ \ \ \ \ \ \
[\mathbb{D},\mathbb{C^+},\mathbb{C^+}]=0,$$
$$ [\mathbb{A},\mathbb{C},\mathbb{B^+}]=0, \ \ \ \ \ \ \
[\mathbb{D},\mathbb{B},\mathbb{C^+}]=0,$$

where $X^+ = \tau X \tau$ (and $\tau$ is the twist map).\\

\end{dfn}

\begin{rem}

Yang--Baxter systems of type II
are related to the
algebras considered in \cite{hlav}, which
include (algebras of functions on) quantum groups, quantum super-groups,
braided groups, quantized braided groups, reflection algebras and others.

\end{rem}

\bigskip

The following theorems present solutions for the Yang--Baxter systems.

\begin{thm} { \bf (see \cite{np2})}
Let $A$ be a commutative algebra and $ \lambda, \lambda' \in \mathbb{K}$.
Then,
$ \ \mathbb{A}, \ \mathbb{B}, \ \mathbb{C}, \ \mathbb{D} : A \otimes A
\rightarrow A \otimes A, $
$ \ \mathbb{A}( a \otimes b)
= \lambda 1 \otimes ab + ab \otimes 1 - b \otimes a $,
$ \ \mathbb{B}( a \otimes b)= \mathbb{C}(a\otimes b)=  1 \otimes ab + ab \otimes 1 - b \otimes a $ and
$ \ \mathbb{D}( a \otimes b)= \lambda' 1 \otimes ab + ab \otimes 1 - b \otimes a $
is a Yang--Baxter system of type II.
\end{thm}

\begin{thm}\label{teoyb1}
Let $ W =  \mathbb{A}, \ X = \mathbb{B} = \mathbb{C}, \
Z = \mathbb{D}$ in the above theorem. It turns out that
$ W, X, Z $ is also a Yang--Baxter system of type I.
\end{thm}

{\bf Proof}

First, let us observe that the result holds even for
$A$ a non-commutative algebra.

One way to prove the theorem is by direct computations.

Alternatively, one can observe that
\begin{equation}   \label{map}
 \psi( a \otimes b)= 1 \otimes ab + ab \otimes 1 - a \otimes b
\end{equation}
is an algebra factorization, and apply Remark 2.4 of \cite{fbrz}.

Also, refer to Theorem 5.2 of \cite{np1}.
\qed

\begin{rem}
One can combine the proof of the Theorem \ref{teoyb1} with Remark 2.4 and
Proposition 2.9 of \cite{fbrz}, to obtain a large class of Yang--Baxter
operators defined on $ V \otimes V$, where $ V= A \oplus A$. See also Remark \ref{remyb1}.
\end{rem}

\begin{thm}\label{thmYBalg}

Let $A$ be an algebra;
$p,q,s,r\in\mathbb{K}$;
$\psi, \psi':A\otimes A\rightarrow A\otimes A$ semi--entwining maps;
$ \ \mathbb{A}, \ \mathbb{B}, \ \mathbb{C}, \ \mathbb{D} : A \otimes A
\rightarrow A \otimes A, $
$ \ \mathbb{A} =R^A_{r,s}$,
$ \ \mathbb{B}= \psi \circ \tau $,
$\mathbb{C} = \psi' \circ \tau $,
$ \ \mathbb{D}=R^A_{p,q}$.
If $\psi'= \tau \circ \psi \circ \tau $, then
$ \ \mathbb{A}, \ \mathbb{B}, \ \mathbb{C}, \ \mathbb{D} $
is a Yang--Baxter system of type II.
\end{thm}

{\bf Proof.} Use Theorem \ref{FlorinBrze}i) to check the first four equations. Then, observe that
$ \mathbb{B}=\mathbb{C^+} \iff \psi'= \tau \circ \psi \circ \tau$. The last four equations then follow.
\qed

\begin{thm}\label{algfactsemithm}

Let $A$ be an algebra and
$\psi:A\otimes A\rightarrow A\otimes A$ a semi--entwining map.

Then, there exists a semi--entwining map
$\psi':A\otimes A\rightarrow A\otimes A$ such that
$\psi'= \tau \circ \psi \circ \tau $ if and only if
$ \psi $, viewed as
$\psi:A^{op}\otimes A\rightarrow A\otimes A^{op}$, is an algebra
factorization.

\end{thm}

{\bf Proof.}
Assume that there exists  a semi--entwining map $\psi'=\tau \circ \psi \circ \tau $.
Denote $\psi'(a\otimes b)=b_{\alpha'}\otimes a^{\alpha'}$, for all $a,b\in A$, i.e.,
$a_{\alpha}\otimes b^{\alpha}=a^{\alpha'}\otimes b_{\alpha'}$.
Also denote by $\opm$ the multiplication in $A^\op$, i.e., for all $a,b\in A$,
$a\opm b\equiv ba$.
Then we must check  conditions of Definition \ref{ent}. For all $a,b,c\in A$,
\begin{align*}
\psi(1_{A^\op}\otimes c)&=\tau\circ\psi'\circ \tau(1_{A^\op}\otimes c)
=\tau\circ\psi'(c\otimes 1_{A^\op})=c\otimes 1_{A^\op},\\
\psi(a\opm b\otimes c)&=\tau\circ\psi'\circ\tau(ba\otimes c)
=\tau\circ\psi'(c\otimes ba)=\tau(b_{\alpha'}a_{\beta'}\otimes c^{\alpha'\beta'})\\
{}&=c^{\alpha'\beta'}\otimes a_{\beta'}\opm b_{\alpha'}
=c_{\alpha\beta}\otimes a^{\beta}\opm b^{\alpha}
\end{align*}

Similarly one can prove the converse.
\qed

\begin{rem} { \bf [Example of algebra factorization for Theorem 3.23]}

We consider the algebra
$ A = A^{op}  = \frac{\mathbb{K}[X]}{(X^2- p)} $, where $ p $ is a  scalar.
Then $A$ has the basis $\{1, x \}$, where $x$ is the image of $X$ in the factor
ring; so, $ x^2 = p $.

If $ q $ is a scalar, then
$\psi :A^{op} \otimes A\rightarrow A\otimes A^{op}$, defined as follows,
\begin{eqnarray*}
&& \psi (1\otimes 1) =  1 \otimes 1 \\
&& \psi (1\otimes x) =  x \otimes 1 \\
&& \psi (x\otimes 1) = 1 \otimes x \\
&& \psi (x\otimes x) = q 1 \otimes 1 -  x \otimes x
\end{eqnarray*}
is an algebra factorization.

Notice that if $ q= 2 p $, then $ \psi $ is the same algebra factorization with
(\ref{map}).

\end{rem}

\bigskip

\begin{thm}  Let $A$ be an algebra, $B$ and $M$  vector spaces, $z \in B \ (z \neq 0)$,
$\psi:B\otimes A\rightarrow A\otimes B$ a semi--entwining, and
let $M$ be an $(A,B,\psi)$-semi--entwined module with the right
measuring $ \phi $.
We consider the maps:

 $X = \psi \circ \tau_{B,A}:B\otimes A\rightarrow B\otimes A$

$\eta : M \otimes A \rightarrow M \otimes A, \ \
m \otimes a \mapsto ma \otimes 1_A$, and

$\zeta : M \otimes B \rightarrow M \otimes B, \ \
m \otimes b \mapsto \phi (m \otimes b) \otimes z$.

Then, the following equation holds:
$$ [\zeta , \eta , X] = 0 $$

\end{thm}

{\bf Proof.}
The proof follows by direct computations. \qed

\begin{rem}
The relation $ [\zeta , \eta , X] = 0 $ from the above theorem is related
to Section 3.6 of \cite{wisbauer}.
\end{rem}

\end{document}